\documentclass{article}
\usepackage{graphicx}
\usepackage{amsmath,amsthm,amssymb,pifont,colortbl,amscd, wrapfig}
\usepackage{bigdelim,multirow}
\newtheorem{theorem}{Theorem}[section]
\newtheorem{lemma}[theorem]{Lemma}
\newtheorem{proposition}[theorem]{Proposition}
\newtheorem{example}[theorem]{Example}
\newtheorem{corollary}[theorem]{Corollary}
\newtheorem{conjecture}[theorem]{Conjecture}

\theoremstyle{definition}

\newtheorem{remark}[theorem]{Remark}
\newtheorem*{acknowledgments}{Acknowledgments}

\def\ev{\mathrm{ev}}

\def\deg{\mathop{\mathrm{deg}}\nolimits}

\def\cl{\mathop{\mathrm{cl}}\nolimits}

\def\tr{\mathop{\mathrm{tr}}\nolimits}

\def\id{\mathop{\mathrm{id}}\nolimits}
\def\ev{\mathop{\mathrm{ev}}\nolimits}

\def\im{\mathop{\mathrm{Im}}\nolimits}
\def\Span{\mathop{\mathrm{Span}}\nolimits}
\def\Span{\mathop{\mathrm{Span}}\nolimits}

\def\co{\colon\thinspace}
\newcommand{\cyc}[2]{ \lfloor \frac{#1}{#2}\rfloor}
\newcommand{\uqenh}[1]{ (\bar U_q^{\ev})\,\hat  {}^{\;\hat  \otimes #1}}
\newcommand{\uqen}[1]{ (\bar U_q^{\ev})\;\tilde {}^{\;\tilde \otimes #1}}
\newcommand{\uqe}{\bar U_q^{\ev}}

\newcommand{\uq}{\bar U_q}
\newcommand{\muq}{\mathcal{ U}_q}
\newcommand{\tmuq}{\tilde{\mathcal{ U}}_q}
\newcommand{\muqe}{\mathcal{ U}_q^{\ev}}
\newcommand{\tmuqe}{\tilde{\mathcal{ U}}_q^{\ev}}
\newcommand{\tmuqen}[1]{(\tilde{\mathcal{ U}}_q^{\ev})^{\tilde \otimes {#1}}}
\newcommand{\uqzq}{U_{\mathbb{Z},q}}
\newcommand{\uqzqe}{U_{\mathbb{Z},q}^{\ev}}

\newcommand{\f}[1]{\tilde F^{({#1})}}
\newcommand{\e}[1]{\tilde E^{({#1})}}
\newcommand{\Z}{\mathbb{Z}[q,q^{-1}]}
\begin{document}
\title{On the universal $sl_2$ invariant of Brunnian bottom tangles}
\author{Sakie Suzuki\thanks{Research Institute for Mathematical Sciences, Kyoto
University, Kyoto, 606-8502, Japan. E-mail address: \texttt{sakie@kurims.kyoto-u.ac.jp}} }
\date{November 27, 2011}
\maketitle
\begin{center}
\textbf{Abstract}
\end{center}
A link $L$ is called Brunnian if every proper sublink of $L$ is trivial. 
Similarly, a bottom tangle $T$ is called Brunnian if every proper subtangle of $T$ is trivial.
In this paper, we give a small subalgebra of the $n$-fold completed tensor power of $U_h(sl_2)$ in which the universal $sl_2$ invariant of $n$-component Brunnian bottom tangles takes values.
As an application, we  
give a divisibility property of  the colored Jones polynomial of  Brunnian links.
\section{Introduction}
The universal invariant of tangles associated with a ribbon Hopf algebra  \cite{H2,He,Ka,KR,R2, R1,O,Re} has the universality property for the colored link invariants which are defined by Reshetikhin and Turaev \cite{Re}.

The  universal $sl_2$ invariant $J_T$ of an $n$-component bottom tangle $T$ takes values in the $n$-fold completed tensor powers $U_h(sl_2)^{\hat \otimes n}$ of $U_h(sl_2)$, and we can obtain the colored Jones polynomial of the closure link $\cl(T)$  from $J_T$ by taking the quantum traces.
Here, a  \textit{bottom tangle} is a tangle in a cube consisting of  only arc components such that each boundary point is on the bottom and the two boundary points of each arc are adjacent to each other,  see  Figure \ref{fig:closure} (a) for example.  The closure of a bottom tangle is defined as in Figure \ref{fig:closure} (b).
\begin{figure}
\centering
\includegraphics[width=9cm,clip]{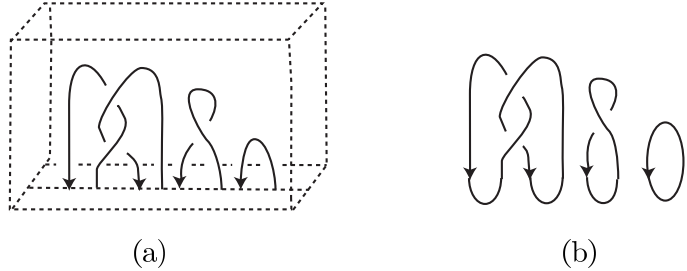}
\caption{ (a) A bottom tangle $T$, (b)  The closure link $\cl(T)$ of $T$}\label{fig:closure}
\end{figure}%
 
Our interest is in the relationship between \textit{topological properties} of tangles and links and
\textit{algebraic properties}  of the universal $sl_2$ invariant and the colored Jones polynomial. 
Habiro \cite{H2} proved that the universal $sl_2$ invariant of  $n$-component,
 algebraically-split, $0$-framed bottom tangles takes values in a  subalgebra $\tmuqen{n}$ of $U_h(sl_2)^{\hat \otimes n}$
(Theorem \ref{h0}).
The present author proved  improvements of  this result with a smaller subalgebra $\uqenh{n}\subset \tmuqen{n}$  in the special case of ribbon bottom tangles \cite{sakie1} and boundary bottom tangles \cite{sakie2} (Theorem \ref{s0}).
Here, the result for boundary bottom tangles had been conjectured by Habiro \cite{H2}.

A link $L$ is called \textit{Brunnian} if  every proper sublink of $L$ is  trivial.
Similarly, a bottom tangle $T$ is called \textit{Brunnian} if every proper subtangle of $T$ is \textit{trivial}, i.e.,  looks like $\cap \cdots \cap$.
Habiro \cite[Proposition 12]{H5} proved that for every Brunnian link $L$, there is a Brunnian bottom tangle whose closure is isotopic to $L$.

In the present paper, we give a subalgebra $U_{Br}^{(n)}$ of $U_h(sl_2)^{\hat \otimes n}$ such that $\uqenh{n}\subset U_{Br}^{(n)}\subset \tmuqen{n}$ in which the universal $sl_2$ invariant of $n$-component  Brunnian bottom tangles takes values (Theorem \ref{1}). 
As an application, we prove a divisibility property of  the colored Jones polynomial of Brunnian links (Theorem \ref{4}).

The rest of this paper is organized as follows.
 In Section \ref{sect1}, we recall basic facts of the quantized enveloping algebra $U_h(sl_2)$.
In Section \ref{sect2}, we define the universal $sl_2$ invariant of bottom tangles.
In Section \ref{sect3}, we give the main result for the universal $sl_2$ invariant of Brunnian bottom tangles.
In Section \ref{sect4}, we give an application for the colored Jones polynomial  of Brunnian links.
Section \ref{proof} is devoted to the proofs of the results.

\section{Quantized enveloping algebra $U_h(sl_2)$}\label{sect1}
In this section, we recall the definition of $U_h(sl_2)$ and its subalgebras.
We follow the notations in \cite{H2, sakie2}.

\subsection{Quantized enveloping algebra $U_h(sl_2)$}
We recall  the definition of the universal enveloping algebra $U_h(sl_2)$.

We denote by  $U_h=U_h(sl_2)$ the $h$-adically complete $\mathbb{Q}[[h]]$-algebra,
topologically generated by  $H, E,$ and $F$, defined by the relations
\begin{align*}
HE-EH=2E, \quad HF-FH=-2F, \quad EF-FE=\frac{K-K^{-1}}{q^{1/2}-q^{-1/2}},
\end{align*}
where we set 
\begin{align*}
q=\exp h,\quad K=q^{H/2}=\exp\frac{hH}{2}.
\end{align*}

We equip $U_h$  with the topological $\mathbb{Z}$-graded algebra structure such that  $\deg E=1$, $\deg F=-1$, and   $\deg H=0$.
For a homogeneous element $x$ of $U_h$, the degree of $x$ is denoted by $|x|$.

\subsection{$\Z$-subalgebras of $U_h(sl_2)$}

We recall $\Z$-subalgebras of $U_h$ from \cite{H2, sakie2}.

In what follows, we use the following $q$-integer notations.
\begin{align*}
&\{i\}_q = q^i-1,\quad  \{i\}_{q,n} = \{i\}_q\{i-1\}_q\cdots \{i-n+1\}_q,\quad  \{n\}_q! = \{n\}_{q,n},
\\
&[i]_q = \{i\}_q/\{1\}_q,\quad  [n]_q! = [n]_q[n-1]_q\cdots [1]_q, \quad \begin{bmatrix} i \\ n \end{bmatrix} _q  = \{i\}_{q,n}/\{n\}_q!,
\end{align*}
for $i\in \mathbb{Z}, n\geq 0$.

Set 
\begin{align}
&\tilde E^{(n)}=(q^{-1/2}E)^n/[n]_q!,\quad \tilde {F}^{(n)}=F^nK^n/[n]_q! \quad \in U_h,\label{deff}
\\
&e=(q^{1/2}-q^{-1/2})E, \quad f=(q-1)FK \quad  \in U_h, \label{defe}
\end{align}
for $n\geq 0$.

Let $U_{\mathbb{Z}, q}\subset U_h$ denote the $\mathbb{Z}[q,q^{-1}]$-subalgebra  generated by
$K,K^{-1}, \tilde E^{(n)}$, and  $\tilde F^{(n)}$ for $n\geq 1$, which is a $\Z$-version of the  Lusztig's  integral  form (cf. \cite{L, sakie1}).

Let $\mathcal U_{q}\subset \uqzq$ denote the $\mathbb{Z}[q,q^{-1}]$-subalgebra generated by
$K,K^{-1}, e$, and  $\tilde F^{(n)}$ for $n\geq 1$. 

Let $\bar {U}_q\subset \mathcal U_{q}$ denote the $\mathbb{Z}[q,q^{-1}]$-subalgebra   generated by
 $K,K^{-1},e$ and $f$, which is a $\Z$-version of the  integral  form defined by De Concini and Procesi (cf. \cite{De, sakie1}).

 For $X=U_{\mathbb{Z}, q}$,  $\muq $, $\bar {U}_q$, let $X^{\ev}$ denote the 
$\mathbb{Z}[q,q^{-1}]$-subalgebra of $U_h$ defined by the same generators as $X$ except that $K^{\pm 2}$ replaces
$K^{\pm 1}$,
i.e., $\uqzqe \subset \uqzq$ denotes the $\mathbb{Z}[q,q^{-1}]$-subalgebra  generated by
$K^2,K^{-2}, \tilde E^{(n)}$, $\tilde F^{(n)}$, $n\geq 1$;
$\muqe\subset \muq$ denotes the $\mathbb{Z}[q,q^{-1}]$-subalgebra  generated by
$K^2,K^{-2}, e$, $\tilde F^{(n)}$, $n\geq 1$; and
$\uqe\subset \uq$ denotes the $\mathbb{Z}[q,q^{-1}]$-subalgebra  generated by
 $K^2,K^{-2},e$, $f$.

To summarize, we have the following inclusions of the subalgebras of $U_h$. 
\begin{center}
\begin{tabular}{cccccccc}
     
    $\uqe $ & $\subset$ & $\muqe$ & $\subset$ & $\uqzqe$ & & \\
     $\cap$ & & $\cap$ & & $\cap$ & &\\
   $\uq$ & $\subset$ &$\muq$ & $\subset$ & $\uqzq$ & $\subset$ &$U_h$  
\end{tabular}
\end{center}

\subsection{Completion}
 In this section, we recall from \cite{H2}  the completion
 $\tmuqe$ of $\muqe$ in $U_h$ and its completed tensor powers $(\tilde {\mathcal{U}}_q^{\ev})^{\tilde {\otimes} n}$  for $n\geq 0$. 
 
First, we define $\tmuq^{\ev}$. For $p\geq 0$, let
$\mathcal{F}_p(\muqe) 
$ be  the two-sided ideal in $\mathcal{U}_q^{\ev}$ generated by $e^p$.
Let $\tilde {\mathcal{U}}_q^{\ev}$ be the completion  of $\mathcal{U}_q^{\ev}$ in $U_h$
with respect to the decreasing filtration $\{\mathcal{F}_p(\mathcal{U}_q^{\ev})\}_{p\geq 0}$, i.e.,  we define
$\tilde {\mathcal{U}}_q^{\ev}$ as the image of the homomorphism
 $$
\varprojlim_{p\geq 0}\mathcal{U}_q^{\ev}/\mathcal{F}_p(\mathcal{U}_q^{\ev}) \rightarrow  U_h
 $$
 induced by $\mathcal{U}_q^{\ev} \subset U_h$.

We define  $(\tilde {\mathcal{U}}_q^{\ev})^{\tilde {\otimes} n}$ for $n\geq 0$.
For $n=0$, we define
$(\tilde {\mathcal{U}}_q^{\ev})^{\tilde \otimes 0}=\mathbb{Z}[q,q^{-1}].$
For $n\geq 1$, we define $(\tilde {\mathcal{U}}_q^{\ev})^{\tilde {\otimes} n}$ as  the completion of $(\mathcal{U}_q^{\ev})^{\otimes n}$ in $U_h^{\hat \otimes n}$ with respect to the decreasing filtration $\{\mathcal{F}_p\big((\mathcal{U}_q^{\ev})^{\otimes n}\big)\}_{p\geq 0}$, where we set 
\begin{align*}
\mathcal{F}_p\big((\mathcal{U}_q^{\ev})^{\otimes n}\big)=\sum_{i=1}^n(\mathcal{U}_q^{\ev})^{\otimes (i-1)}\otimes 
\mathcal{F}_p(\mathcal{U}_q^{\ev})\otimes (\mathcal{U}_q^{\ev})^{\otimes (n-i)}, \quad p\geq 0,
\end{align*}
 i.e.,  we define 
\begin{align*}
(\tilde {\mathcal{U}}_q^{\ev})^{\tilde {\otimes} n}=\im \Big(\varprojlim_{p\geq 0}(\mathcal{U}_q^{\ev})^{\otimes n}/\mathcal{F}_p\big((\mathcal{U}_q^{\ev})^{\otimes n}\big) \rightarrow  U_h^{\hat \otimes n}\Big).
 \end{align*}

For  a $\Z$-subalgebra $A$ of $(\muqe)^{\otimes n}$, $n\geq 0$, we denote by $\{A\}\hat {}$ the \textit{closure} of $A$ in $(\tilde {\mathcal{U}}_q^{\ev})^{\tilde {\otimes} n}$, which is 
the completion  of $A$ in $U_h^{\hat \otimes n}$ with respect to the decreasing filtration $\mathcal{F}_p\big((\mathcal{U}_q^{\ev})^{\otimes n}\big)\cap A$, i.e., 
$$
\{A\}\hat {}=\im \Big(\varprojlim_{p\geq 0}(A/\big(\mathcal{F}_p\big((\mathcal{U}_q^{\ev})^{\otimes n}\big)\cap A \big)\rightarrow  U_h^{\hat \otimes n}\Big).
 $$
In particular, we denoted by $\uqen{n}$ the closure  $\{(\uqe)^{\otimes n}\}\hat {} $ of $(\uqe)^{\otimes n}$ in  $(\tilde {\mathcal{U}}_q^{\ev})^{\tilde {\otimes} n}$.

\section{Universal $sl_2$ invariant of bottom tangles}\label{sect2}
In this section, we recall the definition of the universal $sl_2$ invariant of bottom tangles.
\subsection{Bottom tangles}
A \textit{bottom tangle} (cf. \cite{H1,H2}) is an oriented, framed tangle in a cube  consisting of arc components such that
 each boundary point is on a line  on the bottom,
and  the two boundary points of each component  are adjacent to each other.
We give a preferred orientation of the tangle so that each component runs from its right boundary point to its left boundary point.
For example, see  Figure \ref{fig:closure1} (a), where the dotted lines represent the framing. We draw a diagram of a bottom tangle in a rectangle assuming the blackboard framing, see Figure \ref{fig:closure1} (b).

The \textit{closure link} $\cl(T)$ of a bottom tangle $T$ is defined as the link in $\mathbb{R}^3$ obtained from  $T$ by closing, see Figure \ref{fig:closure} again.
For each $n$-component link $L$, there is an $n$-component bottom tangle whose closure is  $L$.
For a bottom tangle, we can define its linking matrix as that of the closure link.

\begin{figure}
\centering
\includegraphics[width=8.5cm,clip]{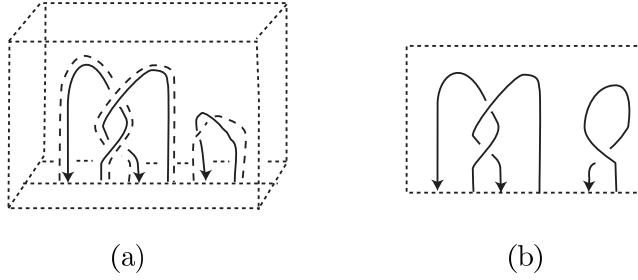}
\caption{ (a) A bottom tangle $T$, (b)  a diagram of $T$}\label{fig:closure1}
\end{figure}%

\subsection{Universal $R$-matrix of $U_h$}
Set 
\begin{align*}
&D=q^{\frac{1}{4}H\otimes H} =\exp \big(\frac{h}{4}H\otimes H\big)\in U_h^{\hat {\otimes }2}.
\end{align*}
We use the following \textit{universal $R$-matrix}  of $U_h$,  
\begin{align*}
R^{\pm 1}=\sum_{n\geq 0}\alpha^{\pm} _n \otimes \beta^{\pm}_n \in U_h^{\hat \otimes 2},
\end{align*}
where we set  formally
\begin{align*}
\alpha _n \otimes \beta _n(&=\alpha^+ _n \otimes \beta^+ _n)=D\Big(q^{\frac{1}{2}n(n-1)}\tilde {F}^{(n)}K^{-n}\otimes e^n\Big),
\\
\alpha _n^- \otimes \beta _n^-&=D^{-1}\Big((-1)^{n}\tilde {F}^{(n)}\otimes K^{-n}e^n\Big).
\end{align*}
(Note that the right hand sides are  sums of infinitely many tensors of the form  $x\otimes y$ with $x,y\in U_h$.
We denote them by $\alpha ^{\pm}_n \otimes \beta^{\pm} _n$ for simplicity.)
\subsection{Universal $sl_2$ invariant of bottom tangles}\label{univinv}

For  an $n$-component bottom tangle $T=T_1\cup \cdots \cup T_n$, we define the universal $sl_2 $ invariant $J_T\in U_h^{\hat {\otimes }n}$  in four steps as follows. We follow the notation in \cite{sakie2}.

\textbf{Step 1. Choose a diagram.}
We choose a diagram $\tilde T$ of $T$ obtained from the copies of the fundamental tangles  depicted in Figure \ref{fig:fundamental},
by pasting horizontally and vertically.  We denote by $C(\tilde T)$ the set of the crossings of $\tilde T$.
For example, for the bottom tangle $B$ depicted  in Figure \ref{fig:base} (a), we can take a diagram $\tilde B$ with $C(\tilde B)=\{c_1,c_2\}$ as depicted in Figure \ref{fig:base} (b).
We call a map 
\begin{align*}
s\co C(\tilde T) \ \ \rightarrow \ \ \{0,1,2,\ldots\}
\end{align*}
a \textit{state}. We denote by  $\mathcal{S}(\tilde T)$ the set of states of the diagram.
\begin{figure}
\centering
\includegraphics[width=9cm,clip]{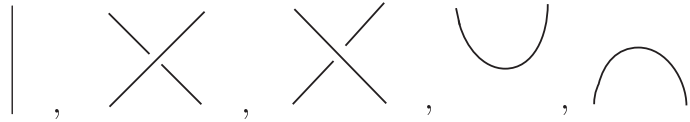}
\caption{Fundamental tangles, where the orientations of the strands are arbitrary
 }\label{fig:fundamental}
\end{figure}%
\begin{figure}
\centering
\includegraphics[width=12cm,clip]{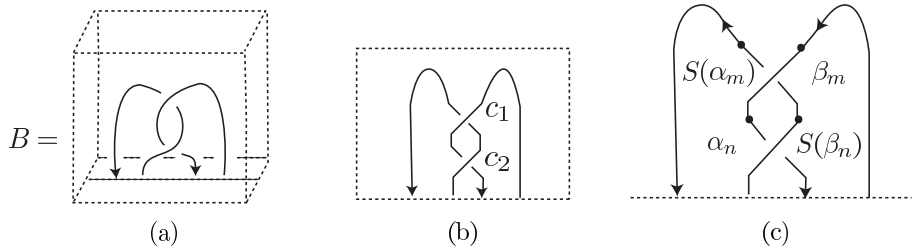}
\caption{(a) A bottom tangle $B$, (b) A  diagram  $\tilde B$ of $B$, (c) The labels associated to a state $t\in \mathcal{S}(B)$
} \label{fig:base}
\end{figure}

\textbf{Step 2.  Attach labels.}
Given a state $s\in \mathcal{S}(\tilde T)$, we attach labels  on the copies of the fundamental tangles in the diagram
following the rule described in Figure \ref{fig:cross}, where  ``$S'$'' should be replaced with $\id$ if 
the string is oriented downward, and with $S$ otherwise.
For example, for a state $t\in \mathcal{S}(\tilde B)$, we put  labels on  $\tilde B$ as in Figure  \ref{fig:base} (c),
where we set $m=t(c_1)$ and $n=t(c_2)$.
\begin{figure}
\centering
\includegraphics[width=12cm,clip]{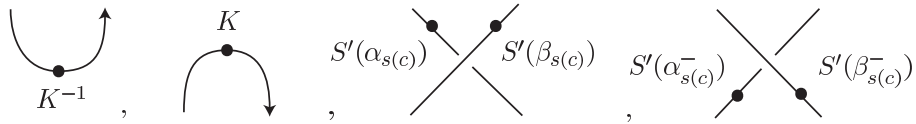}
\caption{How to place labels  on the fundamental tangles}\label{fig:cross}
\end{figure}

\textbf{Step 3. Read the labels.}
We read the labels we have just put on $\tilde T$ and define an element $J_{\tilde T,s}\in U_h^{\hat \otimes n}$ as follows. 
Let $\tilde T=\tilde T_1\cup\cdots  \cup \tilde T_1$, where $\tilde T_i$ corresponds to $T_i$.
We define the $i$th tensorand of $J_{\tilde T,s}$ as the product 
of the labels on $\tilde T_i$, where the labels are read off along $T_i$
reversing the orientation, and   written from left to  right.
For example, for the bottom tangle $B$ and the state $t\in \mathcal{S}(\tilde B)$ in Figure \ref{fig:base}, we have
\begin{align*}
J_{\tilde B,t}&=S(\alpha  _m)S(\beta  _n)\otimes \alpha  _n\beta _m.
\end{align*} 
Here, we identify the labels $S'(\alpha^{\pm} _{i})$ and  $S'(\beta^{\pm} _{i})$
with the first and the second tensorands, respectively, of the element $S'(\alpha^{\pm} _{i})\otimes S'(\beta^{\pm} _{i})\in U_h^{\hat\otimes  2}$.
Also we identify the label $K^{\pm 1}$ with the element  $K^{\pm 1}\in U_h$.
Thus $J_{\tilde T,s}$ is a well-defined element in $U_h^{\hat \otimes n}$. 
For example,  we have
\begin{align*}
J_{\tilde B,t}&=S(\alpha  _m)S(\beta  _n)\otimes \alpha  _n\beta _m
\\
&=\sum q^{\frac{1}{2}m(m-1)} q^{\frac{1}{2}n(n-1)}S(D'_1\tilde {F}^{(m)}K^{-m})S(D''_2e^n)\otimes D'_2\tilde {F}^{(n)}K^{-n}D_1''e^m
\\
&=(-1)^{m+n}q^{-n+2mn}D^{-2}(\tilde F^{(m)}K^{-2n}e^n\otimes \tilde F^{(n)}K^{-2m}e^m)\in U_h^{\hat \otimes 2},
\end{align*}
where $D=\sum D'_1\otimes D''_1=\sum D'_2\otimes D''_2.$
Note that $J_{\tilde T,s}$  \textit{depends} on the choice of the diagram.

\textbf{Step 4.  Take the state sum.}
Set 
\begin{align*}
J_T=\sum_{s\in \mathcal{S}(\tilde T)}J_{\tilde T, s}.
\end{align*}
For example, we have
\begin{align*}
J_B&=\sum_{t\in \mathcal{S}(\tilde B)}J_{\tilde B,t}=\sum_{m,n\geq 0} (-1)^{m+n}q^{-n+2mn}D^{-2}(\tilde F^{(m)}K^{-2n}e^n\otimes \tilde F^{(n)}K^{-2m}e^m).
\end{align*}
As is well known \cite{O},  $J_T$ does not depend on the choice of the diagram, and defines an isotopy invariant of bottom tangles.

\section{Results for the universal $sl_2$ invariant of bottom tangles}\label{sect3}
In this section, we give the main result for the universal $sl_2$ invariant of Brunnian bottom tangles.
\subsection{Universal $sl_2$ invariant of algebraically-split bottom tangles, ribbon bottom tangles and  boundary bottom tangles}
We recall several results for the value of  the universal $sl_2$ invariant of bottom tangles.
Recall the sequence of the subalgebras $\uqe \subset\muqe \subset \uqzqe \subset U_h$.

For an $n$-component bottom tangle $T$, let $\mathrm{Lk}(T)$ denote the linking matrix of $T$.
Set
\begin{align*}
\tilde {D}^{\mathrm{Lk}(T)}&=\prod _{1\leq i\leq n}K_i^{m_{ii}}\prod _{1\leq i\leq j\leq n}D_{ij}^{2m_{ij}}\in U_h^{\hat \otimes n},
\end{align*}
where, $K_i=1^{\otimes i-1} \otimes K \otimes 1^{\otimes n-i}$ for $1\leq i\leq n$ and
\begin{align*}
D_{ij}&=\sum 1^{\otimes i-1} \otimes D' \otimes 1^{\otimes j-i-1} \otimes D'' \otimes 1^{\otimes n-j},
\\
D_{kk}&=\sum 1^{\otimes k-1}\otimes D'D'' \otimes 1^{\otimes n-k},
\end{align*}
for  $1\leq i <j\leq n$, $1\leq k\leq n$,  where  $D=\sum D'\otimes D''$.

\begin{theorem}[{\cite[Proposition 4.2, Remark 4.7]{sakie1}}]\label{WW}
Let  $T$ be  an $n$-component  bottom tangle.
For every diagram $\tilde T$ of $T$ and every  state $s\in \mathcal{S}(\tilde T)$, we have
\begin{align*}
J_{\tilde T,s}\in \tilde D^{\mathrm{Lk}(T)}(\muqe)^{\otimes n}.
\end{align*}
\end{theorem}

More precisely, the proof of  \cite[Proposition 4.2]{sakie1} implies the following.
\begin{proposition}\label{p}
Let  $T$ be an $n$-component  bottom tangle. 
For any diagram $\tilde T$ and  any state $s\in \mathcal {S}(\tilde T)$,
we have
\begin{align*}
J_{\tilde T,s} \in \tilde D^{\mathrm{Lk}(T)} \mathcal{F}_{|s|} ((\muqe)^{\otimes n}),
\end{align*}
where we set $|s|=\max \{ s(c) \ | \ c\in C(\tilde T)\}$.
\end{proposition}

Theorem \ref{WW} and Proposition \ref{p} imply the following.
\begin{theorem}[{\cite[Proposition 4.2, Remark 4.7]{sakie1}}]\label{s3}
For an $n$-component  bottom tangle $T$, we have 
\begin{align*}
J_T\in \tilde D^{\mathrm{Lk}(T)} (\tilde {\mathcal{U}}_q^{ev})^{\tilde \otimes n}.
\end{align*}
\end{theorem}
The following is the special case of Theorem \ref{s3} for algebraically-split  bottom tangle with $0$-framing (i.e., a bottom tangle with $0$-linking matrix), which was proved  first by Habiro \cite{H2}.
\begin{theorem}[{Habiro \cite{H2}}]\label{h0}
Let  $T$ be an $n$-component algebraically-split bottom tangle with $0$-framing.
Then we have
\begin{align*}
J_T\in  (\tilde {\mathcal{U}}_q^{ev})^{\tilde \otimes n}.
\end{align*}
\end{theorem}

In \cite{sakie1} and  \cite{sakie2}, we defined a refined completion $\uqenh{n}\subset \uqen{n}$, and proved the following theorem, which is an improvement of Theorem \ref {h0} in the case of
ribbon bottom tangles and  boundary bottom tangles.

\begin{theorem}[\cite{sakie1, sakie2}]\label{s0}
Let  $T$ be an $n$-component ribbon or boundary bottom tangle with $0$-framing.
Then we have 
\begin{align*}
J_T\in  \uqenh{n}.
\end{align*}
\end{theorem}
\begin{remark}
 Theorem \ref{s0} with $\uqen{n}$ replaced with $\uqenh{n}$ for boundary bottom tangles had been conjectured by Habiro \cite[Conjecture 8.9]{H2}.
Here, we do not know whether the inclusion  $\uqenh{n}\subset \uqen{n}$ is proper or not, but the definition of $\uqenh{n}$ is more natural than that of $\uqen{n}$ in the settings in \cite{sakie1, sakie2}.
\end{remark}

\subsection{Result  for the universal $sl_2$ invariant of Brunnian bottom tangles}
The main result of this paper is the following,  which is an improvement of Theorem \ref {h0} in the case of Brunnian bottom tangles.
\begin{theorem}\label{1}
Let $T$ be an $n$-component Brunnian bottom tangle with $n\geq3$.  We have
\begin{align*}
J_T\in U_{Br}^{(n)},
\end{align*}
where we set 
\begin{align*}
U_{Br}^{(n)}= \bigcap _{i=1}^n \Big\{ \Big((\uqe) ^{\otimes i-1} \otimes  \uqzq ^{\ev}\otimes(\uqe) ^{\otimes n-i}\Big) \cap (\muqe)^{ \otimes n}\Big\}\hat {}.
\end{align*}
\end{theorem}

Here, since a trivial bottom tangle has $0$-framing, a Brunnian bottom tangle also has $0$-framing by the definition.
To compare Theorem \ref{1} with Theorems \ref{h0} and \ref{s0}, for $n\geq 3$, we have the following.
\begin{center}
\begin{tabular}{ccc}
     \{$n$-comp. alg. split bottom tangles with $0$-framing\} &$\stackrel {J}{\rightarrow}$  & $\tmuqen{n}$  \\
    $\cup$ &  & $\cup$ \\
     \{$n$-comp. Brunnian bottom tangles\} & $\stackrel {J}{\rightarrow}$ & $U_{Br}^{(n)}$\\
     &  & $\cup$ \\
   \{$n$-comp. ribbon or boundary bottom tangles with $0$-framing\}  & $\stackrel {J}{\rightarrow}$ & $\uqenh{n}$
\end{tabular}
\end{center}

We can define the Milnor $\mu$ invariants \cite{M1,M2} of a bottom tangle as that of the corresponding string link  described in \cite[Section 13]{H1}.
It is known that the Milnor $\mu$ invariants of ribbon bottom tangles and boundary bottom tangles vanish.
It is also known that the Milnor $ \mu $ invariants of length $\leq n-1$ of $n$-component  Brunnian bottom tangles vanish.
Thus we have the following conjecture.

\begin{conjecture}
\begin{itemize}
\item[(i)] Let $T$ be an $n$-component bottom tangle with $0$-framing.
If the Milnor $\mu$ invariants  of $T$ vanish, then
 we have $J_T\in \uqenh{n}$.

 \item[(ii)]   For $n\geq 3$, let $T$ be an $n$-component bottom tangle with $0$-framing.
If the  Milnor $ \mu $ invariants of $T$ of length $\leq n-1$  vanish,
 then we have $J_T\in U_{Br}^{(n)}.$ 
\end{itemize}
\end{conjecture}

Theorem \ref {1} is derived from the following proposition, which we prove in Section \ref{proof1}. 
\begin{proposition}\label{p1}
Let $T$ be an $n$-component Brunnian bottom tangle with $n\geq3$. For each $i=1,\ldots,  n$,
there is a diagram $\tilde T^{(i)}$ of $T$  such that  
\begin{align*}
J_{\tilde T^{(i)},s}\in (\uqe) ^{\otimes i-1} \otimes  \uqzq ^{\ev}\otimes(\uqe) ^{\otimes n-i}
\end{align*}
for any state  $s\in \mathcal {S}(\tilde T^{(i)})$.
\end{proposition}
\begin{proof}[Proof of Theorem \ref{1} by assuming Proposition \ref{p1}]
For  each $i=1,\ldots ,n$,  by Theorem \ref{1} and Proposition \ref{p}, there is a diagram $\tilde T^{(i)}$ of $T$ such that 
\begin{align*}
J_{\tilde T^{(i)}, s}&\in \Big( (\uqe) ^{\otimes i-1}\otimes \uqzq ^{\ev}\otimes  (\uqe) ^{\otimes n-i}\Big)\cap \mathcal{F}_{|s|} ((\muqe)^{\otimes n})
\end{align*}
for any state  $s\in \mathcal {S}(\tilde T^{(i)})$.
Hence  we have
\begin{align*}
J_{T}&\in \Big \{ \Big( (\uqe) ^{\otimes i-1}\otimes \uqzq ^{\ev}\otimes  (\uqe) ^{\otimes n-i}\Big)\cap (\muqe)^{\otimes n} \Big\}\hat{}
\end{align*}
for all $i=1,\ldots ,n$.
\end{proof}

\begin{example}
For the Borromean bottom tangle $T_B$ depicted in Figure \ref{fig:borromean} (a), we have 
\begin{align*}
J_T &\in   \Big\{ \Big(\uqzqe \otimes  (\uqe) ^{\otimes 2}\Big) \cap (\muqe)^{ \otimes 3}\Big\}\hat {}
\\
&\cap \Big\{ \Big(\uqe \otimes \uqzqe \otimes  \uqe\Big) \cap (\muqe)^{ \otimes 3}\Big\}\hat {}
\\
&\cap \Big\{ \Big((\uqe) ^{\otimes 2} \otimes \uqzqe \Big) \cap (\muqe)^{ \otimes 3}\Big\}\hat {}.
\end{align*}
See Example \ref{ex} for explicit expressions of $J_T$.
\begin{figure}
\centering
\includegraphics[width=8cm,clip]{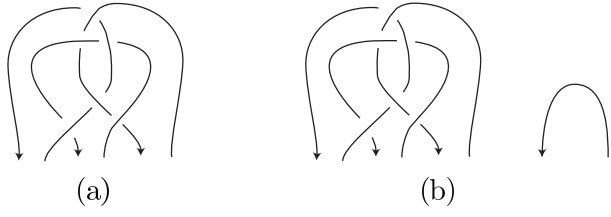}
\caption{(a) The Borromean bottom tangle $T_B$, (b) A bottom tangle $T'_B$}\label{fig:borromean}
\end{figure}%
\end{example}
\begin{example}
Let us add a trivial arc to the Borromean bottom tangle as in Figure \ref{fig:borromean} (b), and denote it by $T_B'$.
Note that  the bottom tangle $T_B'$ is not Brunnian but algebraically-split.
 We have 
\begin{align*}
J_{T_B'}=J_{T_B}\otimes 1\not \in \Big\{ \Big ((\uqe)^{\otimes 3}\otimes \uqzq^{\ev}\Big)\cap  (\muqe)^{ \otimes 4}\Big\}\hat{}.
\end{align*}
\end{example}

\section{Application to the colored Jones polynomial}\label{sect4}
In this section, we give an application  of Theorem  \ref{1} to the colored Jones polynomial of Brunnian links (Theorem \ref{4}).

\subsection{Colored Jones polynomials of algebraically-split links, ribbon links and  boundary links}\label{sect41}
We recall results for the colored Jones polynomials of algebraically-split links, ribbon links, and boundary links.

For $m\geq 1$, let  $V_m$ denote the $m$-dimensional irreducible representation of $U_h$.
Let $\mathcal{R}$  denote the representation ring  of $U_h$ over  $\mathbb{Q}(q^{\frac{1}{2}})$, i.e.,
$\mathcal{R}$ is the $\mathbb{Q}(q^{\frac{1}{2}})$-algebra 
\begin{align*}
\mathcal{R}= \Span _{\mathbb{Q}(q^{\frac{1}{2}})}\{V_m \  | \ m\geq 1\}
\end{align*}
with the multiplication induced by the tensor product.
It is well known that $\mathcal{R}=\mathbb{Q}(q^{\frac{1}{2}})[V_2].$ 

For an $n$-component link $L$ with $0$-framing, take a bottom tangle $T$ whose closure is $L$.
For $X_1,\ldots, X_n\in \mathcal{R}$, the colored Jones polynomial $J_{L; X_1,\ldots , X_n}$ of $L$
with the $i$th component $L_i$ colored by $X_i$ is given by
\begin{align*}
J_{L; X_1,\ldots , X_n}=(\tr_q^{X_1}\otimes \cdots\otimes  \tr_q^{X_n})(J_T)\in \mathbb{Q}(q^{\frac{1}{2}}),
\end{align*}
where, for $Y=\sum_{j}y_jV_j\in \mathcal {R}$ and $u\in U_h$,  
we set 
\begin{align*}
\tr_q^{Y}(u)=\tr^{Y}(K^{-1}u)=\sum_{j}y_j\tr^{V_j}(K^{-1}u).
\end{align*}

Habiro \cite{H2}  studied  the following elements in $\mathcal{R}$
\begin{align}
P_l&=\prod _{i=0}^{l-1}(V_2-q^{i+\frac{1}{2}}-q^{-i-\frac{1}{2}}) \in \mathcal{R},\label{ppr1}
\\
\tilde P'_l&=\frac{q^{\frac{1}{2}l}}{\{l\}_q!}P_l \in \mathcal{R}, \label{ppr2}
\end{align}
for $l\geq 0,$ which are used in an important technical step in his  construction of the unified Witten-Reshetikhin-Turaev invariants for integral homology spheres.

Recall the notation $\{l\}_{q,i} = \{l\}_q\{l-1\}_q\cdots \{l-i+1\}_q$ for $l\in \mathbb{Z}$, $i\geq 0$.
Theorem \ref{h0} implies the following.
\begin{theorem}[Habiro \cite{H2}]\label{h01}
Let $L$ be an $n$-component algebraically-split link with $0$-framing.
For $l_1,\ldots , l_n\geq 0$, we have
\begin{align}\label{z1}
J_{L; \tilde P'_{l_1},\ldots , \tilde P'_{l_n}}\in Z^{(l_1,\ldots, l_n)}_a.
\end{align} 
Here we set
\begin{align*}
Z^{(l_1,\ldots, l_n)}_a&=\frac{\{ 2l_{\max}+1\}_{q, l_{\max}+1}}{\{1\} _q} \Z,
\end{align*}
where $l_{\max}=\max (l_1,\ldots,l_n)$.
\end{theorem}

For $l\geq0,$ let $I_{l}$ denote the ideal in $\mathbb{Z}[q,q^{-1}]$ generated by  $\{l-k\}_q!\{k\}_q!$ for 
$k=0,\ldots, l$.
Theorem \ref{s0}  implies the following improvement of Theorem \ref{h01}. 
\begin{theorem}[\cite{sakie1, sakie2}]\label{s2}
Let $L$ be an $n$-component ribbon or boundary link with $0$-framing.
For $l_1,\ldots , l_n\geq 0$, we have
\begin{align}\label{z2}
J_{L; \tilde P'_{l_1},\ldots , \tilde P'_{l_n}}\in Z^{(l_1,\ldots, l_n)}_{r,b}.
\end{align} 
Here we set
\begin{align*}
Z^{(l_1,\ldots, l_n)}_{r,b}&=  \big(\prod _{1\leq i\leq n, i\neq i_M} I_{l_i} \big)\cdot Z^{(l_1,\ldots, l_n)}_{a}
\\
&=\frac{\{ 2l_{\max}+1\}_{q, l_{\max}+1}}{\{1\} _q}\prod _{1\leq i\leq n, i\neq i_M} I_{l_i},
\end{align*}
 where $l_{\max}=\max (l_1,\ldots,l_n)$ and  $i_M$  is an integer such that   $l_{i_M}=l_{\max}$.

\end{theorem}

For $m\geq 1$, let $\Phi _m=\prod _{d|m}(q^d-1)^{\mu (\frac{m}{d})}\in \mathbb{Z}[q]$ denote the $m$th cyclotomic polynomial, where $\prod _{d|m}$ denotes the
product  over all  positive divisors $d$ of $m$, and  $\mu$ is the M\"obius function.
 For $r\in \mathbb{Q}$, we denote by $\lfloor r \rfloor$  the largest integer smaller than  or  equal to $r$. 

In \cite{sakie3}, we study the ideal $I_l$ and prove the following result, which we use later.
\begin{proposition}[\cite{sakie3}]\label{cy}
For $l \geq 0$, the ideal $I_l$ is  the principal ideal generated by 
\begin{align}\label{lla}
g_l&=\prod_{m\geq 1}\Phi _{m}^{t_{l,m}},
\end{align}
where
\begin{align*}
t_{l,m}&=\begin{cases} 
\cyc{l+1}{m}-1\quad \quad \text{for } 1\leq m\leq l,
\\
0\quad \quad\quad \quad \quad \ \ \text{for } l<m.
\end{cases}
\end{align*}

\end{proposition}

\subsection{Result for the colored Jones polynomial of  Brunnian links}
The following is  an application of Theorem \ref{1} to the colored Jones polynomial of Brunnian links, which we prove in Section \ref{pr2}.
\begin{theorem}\label{4}
Let $L$ be an $n$-component Brunnian link  with $n\geq3$.
For $l_1,\ldots , l_n\geq 0$, we have
\begin{align}\label{z3}
J_{L; \tilde P'_{l_1},\ldots , \tilde P'_{l_n}}&\in Z^{(l_1,\ldots, l_n)}_{Br}.
\end{align}
Here we set
\begin{align*}
Z^{(l_1,\ldots, l_n)}_{Br}&=\frac{\{ 2l_{\max}+1\}_{q, l_{\max}+1}}{\{1\} _q\{l_{\min}\}_q!} \prod _{1\leq i\leq n, i\neq i_M,i_m} I_{l_i},
\end{align*}
where $l_{\max}=\max (l_1,\ldots,l_n)$, $l_{\min}=\min(l_1,\ldots,l_n)$ and $i_M,i_m$, $i_M\neq i_m$, are two integers such that   $l_{i_M}=l_{\max}$, $l_{i_m}=l_{\min}$, respectively.
\end{theorem}

Since a  Brunnian link  $L$ with $n\geq 3$ components is algebraically-split with $0$-framing, 
$L$ satisfies both (\ref{z1}) and (\ref{z3}).
Note that there is no inclusion which satisfies for all $l_1,\ldots , l_n\geq 0$ between  $Z^{(l_1,\ldots, l_n)}_{a}$ and $Z^{(l_1,\ldots, l_n)}_{Br}$.
For example, we have  $Z^{(2,2,2,2)}_{a}\not \subset Z^{(2,2,2,2)}_{Br}$ and  $Z^{(2,2,2,2)}_{Br}\not \subset Z^{(2,2,2,2)}_{a}$
since
\begin{align*}
Z^{(2,2,2,2)}_{a}&=\frac{\{ 5\}_{q, 3}}{\{1\} _q} \Z
\\&=(q-1)^2(q+1)(q^2+q+1)(q^2+1)(q^4+q^3+q^2+q^1+1)\Z,
\\
 Z^{(2,2,2,2)}_{Br}&=\frac{\{ 5\}_{q, 3}}{\{1\} _q\{2\}_q!} \{1\}_q^4 \Z
 \\
 &=(q-1)^4(q^2+q+1)(q^2+1)(q^4+q^3+q^2+q^1+1)\Z.
\end{align*}

For $l_1,\ldots , l_n\geq 0$, set
\begin{align*}
\tilde Z^{(l_1,\ldots, l_n)}_{Br}= Z^{(l_1,\ldots, l_n)}_{a}\cap Z^{(l_1,\ldots, l_n)}_{Br}.
\end{align*}

The above argument implies  the following  refinement of  Theorem \ref{4}.

\begin{theorem}\label{41}
Let $L$ be an $n$-component Brunnian link with $n\geq3$.
For $l_1,\ldots , l_n\geq 0$,
we have
\begin{align*}
J_{L; \tilde P'_{l_1},\ldots , \tilde P'_{l_n}}&\in \tilde Z^{(l_1,\ldots, l_n)}_{Br}.
\end{align*}
\end{theorem} 

For $n\geq3$, we have
\begin{align*}
Z^{(l_1,\ldots, l_n)}_{r,b}= &\big(\prod _{1\leq i\leq n, i\neq i_M} I_{l_i} \big)\cdot Z^{(l_1,\ldots, l_n)}_{a}
\\
=&\big(\{l_{\min}\}_q!I_{l_{\min}}\big)\cdot Z^{(l_1,\ldots, l_n)}_{Br}.
\end{align*}
Thus, comparing Theorem \ref{41} with Theorems \ref{h01} and  \ref{s2}, we have the following for $n\geq 3$.
\begin{center}
\begin{tabular}{ccc}
     \{$n$-comp. alg. split links with $0$-framing\} &$\stackrel {J_{*; \tilde P'_{l_1},\ldots , \tilde P'_{l_n}}}{\longrightarrow}$  & $Z^{(l_1,\ldots, l_n)}_{a}$  \\
    $\cup$ &  & $\cup$ \\
     \{$n$-comp. Brunnian links\} & $\stackrel {J_{*; \tilde P'_{l_1},\ldots , \tilde P'_{l_n}}}{\longrightarrow}$ & $ \tilde Z^{(l_1,\ldots, l_n)}_{Br}$\\
     &  & $\cup$ \\
   \{$n$-comp. ribbon or boundary links with $0$-framing\}  & $\stackrel {J_{*; \tilde P'_{l_1},\ldots , \tilde P'_{l_n}}}{\longrightarrow}$ & $Z^{(l_1,\ldots, l_n)}_{r,b}$
\end{tabular}
\end{center}
\begin{remark}
By Proposition \ref{cy}, the ideals $Z^{(l_1,\ldots, l_n)}_{a}, Z^{(l_1,\ldots, l_n)}_{r,b}$, $Z^{(l_1,\ldots, l_n)}_{Br}$
and  $\tilde Z^{(l_1,\ldots, l_n)}_{Br}$ are principal,  each generated by a product of cyclotomic polynomials. See \cite{sakie3} for details and examples.
\end{remark}
\section{Proofs }\label{proof}
In this section, we prove Proposition \ref{p1} and Theorem \ref{4}.
\subsection{Proof of  Proposition  \ref{p1}}\label{proof1}
We use the following lemma.
\begin{lemma}\label{cl2}
For $m\geq 0$ and $k,l\in \mathbb{Z}$, we have 
\begin{align*}
S^k(\alpha^{\pm } _m)\otimes S^l(\beta^{\pm } _m)\in D^{\pm (-1)^{k+l}}\big((\uqzq\otimes \uq)\cap(\uq\otimes \uqzq)\big),
\end{align*}
\end{lemma}
\begin{proof}
For $m\geq 0$, we have
\begin{align}
\begin{split}
\alpha _m\otimes \beta _m&=D\big(q^{\frac{1}{2}m(m-1)}\tilde {F}^{(m)}K^{-m}\otimes e^m\big)
\\
&=D\big(q^{m(m-1)} f^mK^{-m}\otimes \tilde {E}^{(m)}\big) 
\\
&\in D\big((\uqzq\otimes \uq)\cap(\uq\otimes \uqzq)\big),
\end{split}\label{al1}
\\
\begin{split}
\alpha ^-_m\otimes \beta^- _m&=D^{-1}\big((-1)^{m}\tilde {F}^{(m)}\otimes K^{-m}e^m\big)
\\
&=D^{-1}\big((-1)^{n}q^{\frac{1}{2}m(m-1)}f^m\otimes K^{-m}\tilde {E}^{(m)}\big) 
\\
&\in D^{-1}\big((\uqzq\otimes \uq)\cap(\uq\otimes \uqzq)\big).
\end{split}\label{al2}
\end{align}

For $k,l\in \mathbb{Z}$, we have
\begin{align}
&(S^k\otimes S^l)(D^{\pm 1})=D^{\pm (-1)^{k+l}},
\\
&(S^k\otimes S^l)\big((\uqzq\otimes \uq)\cap(\uq\otimes \uqzq)\big)=(\uqzq\otimes \uq)\cap(\uq\otimes \uqzq).
\end{align}

For $x\in U_h$ homogeneous, we have
\begin{align}\label{ls4}
(x\otimes 1)D^{\pm 1}=D^{\pm 1}(x\otimes K^{\mp |x|}).
\end{align}

Now, (\ref{al1})--(\ref{ls4}) imply the assertion. For example, we have
\begin{align*}
S(\alpha _m)\otimes S(\beta _m)&=(S\otimes S)(\alpha_m \otimes \beta_m)
\\
&\in  (S\otimes S)\Big(D\big((\uqzq\otimes \uq)\cap(\uq\otimes \uqzq)\big)\Big)
\\
&\subset  \big((\uqzq\otimes \uq)\cap(\uq\otimes \uqzq)\big) D
\\
&=D \big((\uqzq\otimes \uq)\cap(\uq\otimes \uqzq)\big).
\end{align*}
\end{proof}

\begin{proof}[Proof of Proposition \ref{p1}]
Let $T=T_1\cup \cdots \cup T_n$ be an $n$-component Brunnian bottom tangle with $n\geq 3$.
We prove the assertion for $i=1$, i.e., we prove that there is a diagram $\tilde T$ of $T$ such that
\begin{align}\label{des}
J_{\tilde T, s}\in  &\uqzqe \otimes \uqe  \otimes \uqe \otimes \cdots\otimes \uqe
\end{align} 
for any state $s\in \mathcal{S}(\tilde T)$.
The other cases $2\leq i\leq n$ are similar.

Since $T$ is Brunnian, the subtangle $T_2\cup \cdots \cup T_n$  is trivial.
Thus $T$ has a diagram $\tilde T=\tilde T_1\cup \tilde T_2\cup  \cdots \cup \tilde T_n$ whose subdiagram $\tilde T_2\cup \cdots \cup \tilde T_n$ has no  crossing.  See Figure \ref{fig:borromeanb} for an example of such a diagram for the Borromean rings $T_B$.
\begin{figure}
\centering
\includegraphics[width=7cm,clip]{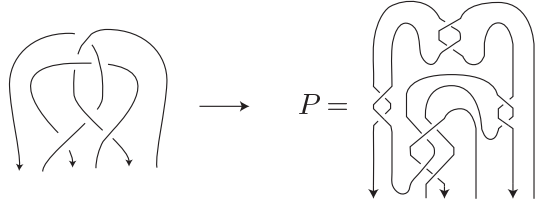}
\caption{Borromean rings $T_B$ and its diagram  $P=P_1\cup P_2 \cup P_3$ such that $P_2 \cup P_3$ has no crossing}\label{fig:borromeanb}
\end{figure}%

We prove that $\tilde T$ satisfies (\ref{des}). 
Note that $\tilde T$ has only  two types of crossings as follows.
\begin{itemize}
\item[] Type A: Crossings between $\tilde T_1$ and $\tilde T_2\cup \cdots \cup \tilde T_n$ 
\item[] Type B: Self crossings of $\tilde T_1$ 
\end{itemize}
Recall from the definition of $J_{\tilde T,s}$ in Section \ref{univinv} the labels which are put on the diagram. 
For the crossings of type A, by Lemma \ref{cl2}, we can assume that the labels on $\tilde T_1$
are legs of copies of $D^{\pm 1}$ and elements of  $\uqzq$, and the labels on $\tilde T_2\cup \cdots \cup \tilde T_n$ are legs of copies of $D^{\pm 1}$ and elements of $\uq$.
For the crossings of type B, we assume that the labels on $\tilde T_1$ are legs of copies of $D^{\pm 1}$ and elements of 
 $\uqzq$.
 See Figure \ref{fig:P} for example, where $\diamond$s denote elements in $\uqzq$ and $\circ$s denote elements in $\uq$.

\begin{figure}
\centering
\includegraphics[width=13cm,clip]{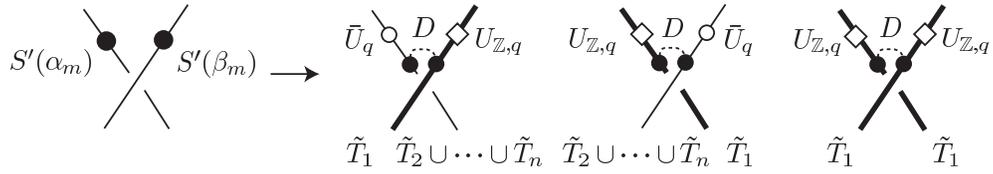}
\caption{The labels on a crossing}\label{fig:P}
\end{figure}%
\begin{figure}
\centering
\includegraphics[width=4cm,clip]{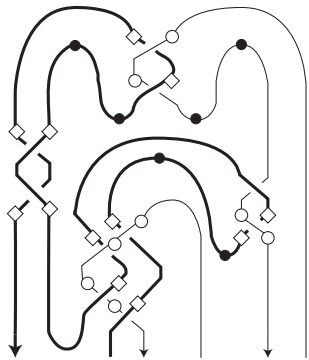}
\caption{Labels except $D^{\pm 1}$s, where the black dots are $K^{\pm1 }$}\label{fig:P2}
\end{figure}%

Now, except copies of  $D^{\pm 1}$, all labels on $\tilde T_1$  are elements of $\uqzq$, and all labels on $\tilde T_2\cup \cdots \cup \tilde T_n$ 
are elements of $\uq$, see Figure \ref{fig:P2}  for example.
We gather every copy of $D^{\pm 1}$ to the leftmost of the expression of $J_{\tilde T,s}$ by using (\ref{ls4}),
and cancel these as $\tilde D^{\mathrm{Lk}(T)}=1$
since the matrix $\mathrm{Lk}(T)$ is $0$. Then we have 
\begin{align*}
J_{\tilde{T}, s}\in \uqzq \otimes \uq \otimes \uq \otimes \cdots\otimes \uq.
\end{align*}
By Theorem \ref{h0}, we have 
\begin{align*}
J_{\tilde{T}, s}\in &\big(\uqzq \otimes \uq \otimes \uq \otimes \cdots\otimes \uq \big)\cap (\muqe)^{\otimes n}
\\
\subset &\uqzqe \otimes \uqe  \otimes \uqe \otimes \cdots\otimes \uqe.
\end{align*}
This completes the proof.
\end{proof}

\begin{example}\label{ex}

The following is the universal $sl_2$ invariant of the Borromean bottom tangle calculated by using the diagram Figure \ref{fig:borromean} (a) (cf. \cite{H2}).

\begin{align*}
J_{T_B}&=\sum_{m_1,m_2,m_3,n_1,n_2,n_3\geq 0} q^{m_3+n_3}(-1)^{n_1+n_2+n_3}q^{\sum _{i=1}^3\big(-\frac{1}{2}m_i(m_i+1)
-n_i +m_im_{i+1}-2m_in_{i-1}\big)}
\\
&\tilde F^{(n_3)}e^{m_1}\tilde F^{(m_3)}e^{n_1}K^{-2m_2}\otimes  \tilde F^{(n_1)}e^{m_2}\tilde F^{(m_1)}e^{n_2}K^{-2m_3}\otimes \tilde  F^{(n_2)}e^{m_3}\tilde F^{(m_2)}e^{n_3}K^{-2m_1}
\\
&\in (\tmuq^{\ev}) ^{\tilde \otimes 3},
\end{align*}
where the index $i$ should be considered modulo $3$.

By using the diagram $P$ in Figure \ref{fig:borromeanb},  after the two self-crossings of the leftmost strand cancel each other out, 
we also have the following expression of $J_{T_B}$.

\begin{align*}
J_{T_B}&=\sum_{g,h,k,l,m,n \geq 0}\sum_{i=0}^m\sum_{j=0}^nt_{g,h,i,j,k,l,m,n}(q)
\\
&K^{-2(h+k)}\f{g}\e{h}\e{j}\f{i}\e{k}\f{l}\f{m-i}\e{n-j}\otimes K^{2(k-l-m)}f^{n}e^{m}\otimes K^{-2(h-i+j+k)}e^{l}f^{h+k}e^g
\\
&\in   \Big\{ \Big(\uqzqe \otimes  (\uqe) ^{\otimes 2}\Big) \cap (\muqe)^{ \otimes 3}\Big\}\hat {},
\end{align*}
where
\begin{align*}
t_{g,h,i,j,k,l,m,n}(q)=(-1)^{g+h+m+n+i+j}&q^{-2g(3h+k)+\frac{1}{2}h(h-1)+h(2l-1)+k(2l+2n-i-j-1)-\frac{1}{2}l(l-1)}
\\
&\cdot q^{-l(2n+i-3j)-m(6n-i-j)+\frac{1}{2}n(n-1)-n(j+1)-\frac{1}{2}i(i-1)+\frac{1}{2}j(j-1)}.
\end{align*}
\end{example}
\subsection{Proof of Theorem \ref{4}}\label{pr2}
In this section, we  prove  Theorem \ref{4}.

First of all, we recall generators of $\uqe$ and $\uqzqe$ as $\mathbb{Z}[q,q^{-1}]$-modules.
The following Lemma is a  variant of a well-known fact about the integral form of
De Concini and Procesi (cf. \cite{De, H2}).
\begin{lemma}\label{F2}
 $ \bar {U}_q^{\ev}$ is freely $\mathbb{Z}[q,q^{-1}]$-spanned by the elements $f^iK^{2j} e^k$ with $i,k\geq 0$ and $j\in \mathbb{Z}$. 
\end{lemma}

Set
\begin{align*}
\begin{bmatrix} H \\s\end{bmatrix}_q=&\frac{(K^2-1)(q^{-1}K^2-1)\cdots (q^{-s+1}K^2-1)}{\{s\}_q!}
\\
=&\frac{1}{\{s\}_q!}\sum_{p=0}^{s}(-1)^{s-p}q^{\frac{1}{2}p(p+1)-sp}\begin{bmatrix} s \\p\end{bmatrix}_qK^{2p}
\end{align*}
for  $s\geq 0$.

The following Lemma is a  variant of a well-known fact about 
Lusztig's integral form (cf. \cite{L}).

\begin{lemma}\label{F1}
 $\uqzqe$ is  $\mathbb{Z}[q,q^{-1}]$-spanned by the elements $\f{i}K^{2j}\begin{bmatrix} H\\s\end{bmatrix}_q \e{k}$ with $i,k,s\geq 0$ and $j\in \mathbb{Z}$. 
\end{lemma}

For the elements $P_l, \tilde P'_l\in \mathcal{R}$ defined in (\ref{ppr1}),  (\ref{ppr2}) in Section \ref{sect41}, 
 we have the following results.

 \begin{lemma}[{Habiro \cite[Lemma 8.8]{H2}}]\label{ls0}
 \begin{itemize}
\item[\rm{(1)}]
If $l,i,i'\geq 0$, $i\not = i'$, and $j\in \mathbb{Z}$, then we have $\tr_q^{P_l}(\f{i}K^{2j}e^{i'})=0$.

\item[\rm{(2)}] 
If $0 \leq l < i$ and $j\in \mathbb{Z}$, then we have $\tr_q^{P_l}(\f{i}K^{2j}e^{i})=0$.

\item[\rm{(3)}]
If $0\leq i\leq l$ and $j\subset \mathbb{Z}$, then we have
\begin{align*}
\tr_q^{P_l}(\f{i}K^{2j}e^{i})=q^{\frac{1}{2}l-lj+2ij+i^2-il}\{l\}_q!\{l-i\}_q!\begin{bmatrix}
j+l-1 \\
l-i
\end{bmatrix}_q\begin{bmatrix}
j-1 \\
l-i
\end{bmatrix}_q.
\end{align*}
\end{itemize}
 \end{lemma}

For $l\geq0,$ recall the ideal $I_{l}$  in $\mathbb{Z}[q,q^{-1}]$, which is  generated by  $\{l-i\}_q!\{i\}_q!$ for 
$i=0,\ldots, l$.
 \begin{corollary}[Habiro \cite{H2}]\label{ls2}
For $l\geq 0$, we have $\tr_q^{\tilde P'_l}(\uqe)\subset I_{l}$.
\end{corollary}
\begin{proof}
The assertion follows from Lemma \ref{F2}, Lemma \ref{ls0} (1), (2),  and 
\begin{align*}
\tr_q^{\tilde P'_l}(f^iK^{2j}e^{i})=&q^{-\frac{1}{2}i(i-1)}\{i\}_q!\tr_q^{\tilde P'_l}(\f{i}K^{2j}e^{i})
\\
=&q^{-\frac{1}{2}i(i-1)}\{i\}_q!\frac{q^{\frac{1}{2}l}}{\{l\}_q!}\tr_q^{ P_l}(\f{i}K^{2j}e^{i})
\\
=&q^{-\frac{1}{2}i(i-1)+l-lj+2ij+i^2-il}\{i\}_q!\{l-i\}_q!\begin{bmatrix}
j+l-1 \\
l-i
\end{bmatrix}_q\begin{bmatrix}
j-1 \\
l-i
\end{bmatrix}_q \in I_l
\end{align*}
for $0\leq i\leq l$ and $j\in \mathbb{Z}$.
\end{proof}

\begin{proposition}\label{ls}
For $l\geq 0$, we have
$\{l\}_q! \tr_q^{\tilde P'_{l}}(\uqzq^{\ev})\in \Z $.
\end{proposition}
\begin{proof}
By Lemma \ref{F1} and Lemma \ref{ls0} (1), (2), it is enough to check
\begin{align*}
\{l\}_q! \tr_q^{\tilde P'_{l}}(\f{i}K^{2j}\begin{bmatrix} H\\s\end{bmatrix}_q\e{i})\in \Z
\end{align*}
for $0\leq i\leq l$, $s\geq 0$ and $j\in \mathbb{Z}$.

For $s=0$, by Lemma \ref{ls2} (3), we have
\begin{align}
\begin{split}\label{ari}
\{l\}_q! \tr_q^{\tilde P'_l}(\f{i}K^{2j}\e{i})=&\frac{\{l\}_q! }{\{i\}_q!}\tr_q^{\tilde P'_l}(\f{i}K^{2j}e^{i})
\\
=&\frac{q^{\frac{1}{2}l}}{\{i\}_q!}\tr_q^{ P_l}(\f{i}K^{2j}e^{i})
\\
=&q^{l-lj+2ij+i^2-il}\{l\}_{q,l-i}\{l-i\}_q! \begin{bmatrix}
j+l-1 \\
l-i
\end{bmatrix}_q\begin{bmatrix}
j-1 \\
l-i
\end{bmatrix}_q.
\end{split}
\end{align}
For $s\geq 0$, by using the case $s=0$, we have
\begin{align*}
\{l\}_q!& \tr_q^{\tilde P'_l}(\f{i}K^{2j}\begin{bmatrix} H \\s\end{bmatrix}_q\e{i})
\\
=&\{l\}_q! \frac{1}{\{s\}_q!}\sum_{p=0}^{s}(-1)^{s-p}q^{\frac{1}{2}p(p+1)-sp}\begin{bmatrix} s \\p\end{bmatrix}_q\tr_q^{\tilde P'_l}(\f{i}K^{2j+2p}\e{i})
\\
=& \frac{1}{\{s\}_q!}\sum_{p=0}^{s}(-1)^{s-p}q^{\frac{1}{2}p(p+1)-sp}\begin{bmatrix} s \\p\end{bmatrix}_qq^{l-lj+2i(j+p)+i^2-il}\{l\}_{q,l-i}\{l-i\}_q! \begin{bmatrix}
(j+p)+l-1 \\
l-i
\end{bmatrix}_q\begin{bmatrix}
(j+p)-1 \\
l-i
\end{bmatrix}_q
\\
=& \frac{1}{\{s\}_q!}\sum_{p=0}^{s}(-1)^{s-p}q^{\frac{1}{2}p(p+1)-sp}\begin{bmatrix} s \\p\end{bmatrix}_qq^{l-lj+2i(j+p)+i^2-il}\begin{bmatrix}
l \\
l-i
\end{bmatrix}_q \{(j+p)+l-1\}_{q,l-i}
\{(j+p)-1\}_{q,l-i}
\\
=& \frac{1}{\{s\}_q!}\sum_{p=0}^{s}(-1)^{s-p}q^{\frac{1}{2}p(p+1)-sp}\begin{bmatrix} s \\p\end{bmatrix}_qq^{l-lj+2i(j+p)+i^2-il}\begin{bmatrix}
l \\
l-i
\end{bmatrix}_q 
\\
&\Big(\sum_{t=0}^{l-i}(-1)^{l-i-t}q^{\frac{1}{2}t(t+1)+(j+p+i-1)t}\begin{bmatrix} l-i \\t\end{bmatrix}_q\Big)
\Big(\sum_{u=0}^{l-i}(-1)^{l-i-u}q^{\frac{1}{2}u(u+1)+(j+p-l+i-1)u}\begin{bmatrix} l-i \\u\end{bmatrix}_q\Big)
\\
=& \sum_{t=0}^{l-i}\sum_{u=0}^{l-i}(-1)^{t+u}q^{l-lj+2ij+i^2-il+\frac{1}{2}t(t+1)+(j+i-1)t+\frac{1}{2}u(u+1)+(j-l+i-1)u}\\
&\Big(\frac{1}{\{s\}_q!}\sum_{p=0}^{s}(-1)^{s-p}q^{\frac{1}{2}p(p+1)+(i+t+u-s)p}\begin{bmatrix} s \\p\end{bmatrix}_q\Big)
\begin{bmatrix}
l \\
l-i
\end{bmatrix}_q 
\begin{bmatrix} l-i \\t\end{bmatrix}_q
\begin{bmatrix} l-i \\u\end{bmatrix}_q
\\
=& \sum_{t=0}^{l-i}\sum_{u=0}^{l-i}(-1)^{t+u}q^{l-lj+2ij+i^2-il+\frac{1}{2}t(t+1)+(j+i-1)t+\frac{1}{2}u(u+1)+(j-l+i-1)u}\\
&\begin{bmatrix} i+t+u \\s\end{bmatrix}_q\begin{bmatrix}
l \\
l-i
\end{bmatrix}_q 
\begin{bmatrix} l-i \\t\end{bmatrix}_q
\begin{bmatrix} l-i \\u\end{bmatrix}_q\quad \in \Z.
\end{align*}
Here, the first equality follows from the definition of $\begin{bmatrix} H \\s\end{bmatrix}_q$, 
the second equality follows from (\ref{ari}), and the other equalities follow from straightforward calculations, where we use
\begin{align*}
\{k\}_{q,n}=\sum_{r=0}^n(-1)^{n-r}q^{\frac{1}{2}r(r+1)+r(k-n)}
\begin{bmatrix}n \\r \end{bmatrix}_q
\end{align*}
for $k\in \mathbb{Z}$ and $n\geq0$.
Hence we have the assertion.
\end{proof}
We use the following proposition.
 \begin{proposition}\label{iin}
 Let $T$ be an  $n$-component Brunnian bottom tangle with $n\geq 3$. For $1\leq i\leq n$ and $l_i\geq 0,$ we have 
 \begin{itemize}
 \item[\rm{(i)}] $(\id^{\otimes i-1} \otimes \tr_q^{\tilde P'_{l_i}}\otimes\id^{\otimes n-i})(J_T)\in (\muqe) ^{\otimes n-1},$
   
 \item[\rm{(ii)}] $\{l_i\}_q!(\id^{\otimes i-1} \otimes \tr_q^{\tilde P'_{l_i}}\otimes\id^{\otimes n-i})(J_T)\in (\uqe) ^{\otimes n-1}.$
 \end{itemize}
 
 \end{proposition}
  \begin{proof} We prove the assertion with $i=1$. The other cases are similar.
 Let $\tilde T=\tilde T^{(1)}$ be a diagram of $T$ as in Proposition \ref{p1}.
 By the proof of Proposition \ref{p1}, we can assume that $\tilde T=\tilde T_1\cup \cdots \cup \tilde T_1$ has only  two types of crossings as follows.
\begin{itemize}
\item[] Type A: Crossings between $\tilde T_1$ and $\tilde T_2\cup \cdots \cup \tilde T_n$ 
\item[] Type B: Self crossings of $\tilde T_1$ 
\end{itemize}

Let $s\in \mathcal{S}(\tilde T)$.  Set $|s|=\max \{s(c) \ | \ c\in C(\tilde T)\}.$
Note that, for $0 \leq m<p$, the elements $E^p$ and $F^p$ act as $0$ on the  $m$-dimensional irreducible representation $V_m$ of $U_h$. 
Since each crossing of either type involves the strand $\tilde{T}_1$, there is a crossing $c$ on $\tilde T_1$ such that  $s(c)=|s|$.
Since $\tilde P'_{l_1}\in \Span _{\mathbb{Q}(q^{1/2})}\{V_0,\ldots, V_{l_1}\}$,
if  $|s|> l_1$,  we have 
\begin{align}
 (\tr_q^{\tilde P'_{l_1}}\otimes\id^{\otimes n-1})(J_{\tilde T,s})=0. \label{itt}
  \end{align}
By (\ref{itt}), Theorem \ref{WW} with $\mathrm{Lk}(T)=0$ implies (i), and  Propositions \ref{p1}, \ref{ls} imply (ii).
  \end{proof}
  
 For a subalgebra $X$ of $U_h$, let $Z(X)$ denote the center of $X$.
Habiro \cite[Proposition 8.6]{H2} proved that for an $n$-component algebraically-split bottom tangle with $0$-framing, we have
\begin{align*}
 (\id \otimes \tr_q^{\tilde P'_{l_2}}\otimes \tr_q^{\tilde P'_{l_3}}\otimes \cdots \otimes \tr_q^{\tilde P'_{l_n}})(J_T)
 \in Z(\tmuqe).
 \end{align*}
 We can improve this result for Brunnian bottom tangles as follows.

 \begin{proposition}\label{hc}
 Let $T$ be an  $n$-component Brunnian bottom tangle with $n\geq 3$. For $l_2,\ldots, l_n\geq 0,$ we have 
 \begin{align*}
 (\id \otimes \tr_q^{\tilde P'_{l_2}}\otimes \tr_q^{\tilde P'_{l_3}}\otimes \cdots \otimes \tr_q^{\tilde P'_{l_n}})(J_T)
 \in Z(\muqe).
  \end{align*}
  
 \end{proposition}
 \begin{proof}
By Proposition  \ref{iin} (i)  and $\tr_q^{\tilde P'_{l}}(\muqe)\subset \Z$ for $l\geq 0$, we have
 \begin{align*}
 (\id \otimes \tr_q^{\tilde P'_{l_2}}\otimes \tr_q^{\tilde P'_{l_3}}\otimes \cdots \otimes \tr_q^{\tilde P'_{l_n}})(J_T)
 &\in (\id \otimes \tr_q^{\tilde P'_{l_3}}\otimes \cdots \otimes \tr_q^{\tilde P'_{l_n}})((\muqe)^{\otimes n-1})
 \\
 &\subset \muqe.
\end{align*}
 It is well-known that $J_T$ is contained in the invariant part of $U_h^{\hat \otimes n}$ (cf. \cite[Proposition 4.2]{H2}). This fact implies
\begin{align*}
(\id \otimes \tr_q^{\tilde P'_{l_2}}\otimes \tr_q^{\tilde P'_{l_3}}\otimes \cdots \otimes \tr_q^{\tilde P'_{l_n}})(J_T)
 &\in  Z(U_h).
\end{align*}
Since $\muqe \cap Z(U_h)\subset Z(\muqe)$, we have the assertion.
 \end{proof}
 
Let $C=(q^{1/2}-q^{-1/2})^2FE+q^{1/2}K+q^{-1/2}K^{-1}$ denote the Casimir element.
 Recall from \cite{H2} that $Z(\muq^{\ev})$ is freely generated by $C^2$ as a $\Z$-algebra, and thus, freely spanned by the 
 following monic polynomials in $C^2$  as a $\Z$-module.
\begin{align*}
\sigma_p= \prod _{i=1}^p\big(C^2-(q^i+2+q^{-i})\big), \quad p\geq0.
\end{align*}

Habiro proved the following.
 
\begin{proposition}[Habiro {\cite[Proposition 6.3]{H2}}]\label{h3}
For $l,m \geq 0,$ we have
\begin{align*}
\tr_q^{P''_l}(\sigma_m)=\delta_{l,m},
\end{align*}
where $P_l''=q^{l(l+1)}\frac{\{1\}_q}{\{2l+1\}_{q,l+1}}\tilde P'_l$.
\end{proposition} 
 Proposition \ref{h3}  implies the following.
 \begin{corollary}[Habiro \cite{H2}]\label{ls3}
 For $l \geq 0,$ we have
\begin{align*}
\tr_q^{\tilde P'_{l}}\big(Z(\muqe)\big)\subset \frac{\{2l+1\}_{q,l+1}}{\{1\}_q}\Z.
\end{align*}
 \end{corollary}

 Now, we prove Theorem \ref{4}.
\begin{proof}[Proof of Theorem \ref{4}]
For $n\geq 3$, let $L$ be an $n$-component  Brunnian link and $T$ a Brunnian bottom tangle whose closure is $L$.
Let $l_1,\ldots ,l_n\geq 0$.
Without loss of generality, we assume $l_1=\max (l_1,\ldots,l_n)$ and $l_2=\min\{l_i \ | \ 1\leq i\leq n\}$.
By   Proposition \ref{iin} (ii) and  Corollary \ref{ls2}, we have
\begin{align}
\begin{split}\label{la3}
&\{l_2\}_q!(\id\otimes \tr_q^{\tilde P'_{l_2}}\otimes \tr_q^{\tilde P'_{l_3}}\otimes \cdots \otimes \tr_q^{\tilde P'_{l_n}})(J_ T)
\\
&\in(\id\otimes \tr_q^{\tilde P'_{l_3}}\otimes \tr_q^{\tilde P'_{l_3}}\otimes \cdots \otimes \tr_q^{\tilde P'_{l_n}})
((\uqe)^{\otimes n-1})
\\
&\subset \big(\prod_{3\leq i\leq n} I_{i}\big)\cdot \uqe
\\
&\subset \big(\prod_{3\leq i\leq n} I_{i}\big)\cdot \muqe
\\
&=g_{l_3}\cdots g_{l_n} \muqe,
\end{split}
\end{align}
where the last equation is follows from  Proposition \ref{cy}.

Since $\muqe$ has no non-trivial zero divisor, we have
\begin{align}\label{la4}
\big(g_{l_3}\cdots g_{l_n}\muqe\big)\cap Z(\muqe)&\subset  g_{l_3}\cdots g_{l_n}Z(\muqe).
\end{align}
By (\ref{la3}), (\ref{la4}) and Proposition \ref{hc}, we have
\begin{align}\label{la5}
\{l_2\}_q!(\id\otimes \tr_q^{\tilde P'_{l_2}}\otimes \tr_q^{\tilde P'_{l_3}}\otimes \cdots \otimes \tr_q^{\tilde P'_{l_n}})(J_{T})
\subset g_{l_3}\cdots g_{l_n}Z(\muqe).
\end{align}
By (\ref{la5}) and Corollary \ref{ls3}, we have 
\begin{align*}
\{l_2\}_q!J_{L; \tilde P'_{l_1},\ldots , \tilde P'_{l_n}}&=\{l_2\}_q!(\tr_q^{\tilde P'_{l_1}}\otimes \tr_q^{\tilde P'_{l_2}}\otimes \tr_q^{\tilde P'_{l_3}}\otimes \cdots \otimes \tr_q^{\tilde P'_{l_n}})(J_{T})
\\
&\in \tr_q^{\tilde P'_{l_1}}\big(g_{l_3}\cdots g_{l_n}Z(\muqe)\big)
\\
&\subset \frac{\{2l_1+1\}_{q,l_1+1}}{\{1\}_q}g_{l_3}\cdots g_{l_n}\Z
\\
\\&= \frac{\{2l_1+1\}_{q,l_1+1}}{\{1\}_q}\prod_{3\leq i\leq n}I_{l_i}.
\end{align*}
Hence we have 
\begin{align*}
J_{L; \tilde P'_{l_1},\ldots , \tilde P'_{l_n}}&= \frac{\{2l_1+1\}_{q,l_1+1}}{\{1\}_q\{l_2\}_q!}\prod_{3\leq i\leq n}I_{l_i}.
\end{align*}
This completes the proof.
\end{proof}

\begin{acknowledgments}
This work was partially supported by JSPS Research Fellowships for Young Scientists.
The author is deeply grateful to Professor Kazuo Habiro and Professor Tomotada Ohtsuki
for helpful advice and encouragement.
\end{acknowledgments}

\end{document}